\documentclass[10pt,a4paper]{amsart}
\usepackage{amsmath,amsthm,amssymb}
\usepackage[applemac]{inputenc}
\usepackage{graphicx}

\usepackage{amsfonts}
\usepackage{epsf}
\theoremstyle{plain}
\newtheorem{teo}{Theorem}[section]
\newtheorem{prop}[teo]{Proposition}
\newtheorem{coro}[teo]{Corollary}
\newtheorem{lm}[teo]{Lemma}

\theoremstyle{definition}
\newtheorem{exa}[teo]{Example} 
\newtheorem{oss}[teo]{Remark}

\newtheorem*{ack}{Acknowledgements}
\numberwithin{equation}{section}

\DeclareMathOperator*{\osc}{osc}
\def\ep{\varepsilon}
\def\SSN{{\mathbb S}^{N-1}}

\def\RN{\mathbb{R}^N}
\def\RR{\mathbb{R}}

\def\cA{\mathcal{A}}
\def\cC{\mathcal{C}}

\def\cS{\mathcal{S}}
\def\al{\alpha}
\def\la{\lambda}
\def\cP{\mathcal P}

\def\pa{\partial}

\def\shk{\cS_\om(\cK)}

\def\dist{\mbox{\rm dist}}

\def\diam{\mbox{\rm diam}}

\newcommand{\cK}{\mathcal K}
\newcommand{\wh}{\widehat}
\newcommand{\De}{\Delta}

\newcommand{\cT}{\mathcal T}

\newcommand{\cR}{\mathcal R}
\newcommand{\cM}{\mathcal M}
\newcommand{\sums}{\sum\limits}

\newcommand{\core}{\heartsuit}

\newcommand{\si}{\sigma}
\newcommand{\Om}{\Omega}
\newcommand{\om}{\omega}
\newcommand{\acc}{\`}

\newcommand{\RE}{\mathbb R}
\newcommand{\na}{\nabla}
\newcommand{\ovr}{\overline}

\newcommand{\ds}{\displaystyle}
\newcommand{\omp}{\om^\perp}
\newcommand{\one}{\mathcal{X}}

\title[The location of the hot spot]{The location of the hot spot in a grounded convex conductor}
\author{Lorenzo Brasco}
\address{Dipartimento di Matematica e Applicazioni ``R. Caccioppoli'', Universit\`a di Napoli ``Federico II'', Complesso di Monte S. Angelo, Via Cintia, 80126 Napoli, Italy}
\email{lorenzo.brasco@unina.it}
\author{Rolando Magnanini}
\author{Paolo Salani}
\address{Dipartimento di Matematica ``U. Dini'', Universit\acc a di Firenze, viale Morgagni 67/A, 50134 Firenze, Italy}
\email{magnanin@math.unifi.it, paolo.salani@math.unifi.it}

\begin{document}

\maketitle
\begin{abstract}We investigate the location of 
the (unique) hot spot in a convex heat conductor with unitary initial
temperature and with boundary grounded at zero temperature. We
present two methods to locate the hot spot: the former is based
on ideas related to the Alexandrov-Bakelmann-Pucci maximum principle and  
Monge-Amp\` ere equations; the latter relies on 
Alexandrov's reflection principle.
We then show how such a problem can be simplified in case the conductor 
is a polyhedron. Finally, we present some numerical computations.
\end{abstract}

\section{Introduction}

Consider a heat conductor $\Om$ having (positive) constant initial temperature while its boundary is
constantly kept at zero temperature. This physical situation can be described
by the following initial-boundary value problem for heat equation:
\begin{equation}
\label{heatconductor}
\begin{array}{ll}
u_t=\Delta u\quad\mbox{ in }\ &\Om\times (0,\infty),\\
u=1 \quad\quad\mbox{ on } &\Om\times\{ 0\},\\
u=0 \quad\quad\mbox{ on } &\pa\Om\times
(0,\infty).
\end{array}
\end{equation}
Here $\Om$ --- the {\it heat conductor} --- is a bounded domain
in the Euclidean space $\RE^N,$ $N\ge 2,$ with Lipschitz boundary and $u=u(x,t)$ denotes the normalized temperature
of the conductor at a point $x\in\Om$ and time $t>0.$
\par
A {\it hot spot} $x(t)$ is a point at which the temperature $u$ attains 
its maximum at each given time $t,$ that is such that
$$
u(x(t),t)=\max_{y\in\ovr{\Om}} u(y,t).
$$
If $\Om$ is  convex (in this case $\ovr{\Om}$ is said a {\it convex body}), it is well-known by a result of Brascamp and Lieb \cite{BL} that
$\log u(x,t)$ is concave in $x$ for every $t>0$ and this, together with the analyticity of $u$ in $x,$
implies that for every $t>0$ there is a unique point $x(t)\in\Om$ at which the gradient
$\na u$ of $u$ vanishes (see also \cite{MS1}).
\par
The aim of this paper is to give quantitative information on the hot spot's location in a convex body.
\par
A description of the evolution with time of the hot spot can be found in 
\cite{Sa}; we summarize it here for the reader's convenience. 
A classical result of Varadhan's \cite{Va} tells us where $x(t)$ is located for small times:
since
$$
-4t\,\log\{ 1-u(x,t)\}\to\dist(x,\pa\Om)^2 \ \mbox{ uniformly for $x\in\ovr{\Om}$ as $t\to 0^+$}
$$
(here $\dist(x,\pa\Om)$ is the distance of $x$ from $\pa\Om$),
we have that
\[
\dist(x(t),\cM)\to 0 \ \mbox{ as $t\to 0^+$},
\]
where
\begin{equation}
\label{inithotset}
\cM=\{x\in\Om: \dist(x,\pa\Om)=r_\Omega\}
\end{equation}
and
$$
r_\Omega=\max\{\dist(y,\pa\Om)\,:\,y\in\ovr{\Om}\}
$$
is the {\em inradius} of $\Omega$. In particular, we have that 
\begin{equation}
\label{short}
\dist(x(t),\pa\Om)\to r_\Om \ \mbox{ as $t\to 0^+$},
\end{equation}
For large times instead, we know that $x(t)$ must be close to the maximum point $x_\infty$ of the
first Dirichlet eigenfunction $\phi_1$ of $-\De.$ Indeed, denoting with $\la_1=\la_1(\Om)$
the eigenvalue corresponding to $\phi_1$, we have that $e^{\la_1 t} u(\cdot,t)$
converges to $\phi_1$ locally in $C^2$ as $t$ goes to $\infty$;
therefore (see \cite{Sa})
\begin{equation}
\label{large}
x(t)\to x_\infty \ \mbox{ as $t\to\infty$}.
\end{equation}
\par
While it is relatively easy to locate the set $\cM$ by geometrical means, \eqref{large}
does not give much information: locating either $x(t)$ or $x_\infty$ has more or less
the same difficulty.
In this paper, we shall develop geometrical means to estimate the location of
$x(t)$ (or $x_\infty$), based on two kinds of arguments.
\par
The former is somehow
reminiscent of the proof of the maximum principle of Alexandrov, Bakelmann and Pucci and of some ideas contained in \cite{Ta}, concerning properties of
solutions of the Monge-Amp\` ere equation. The estimates obtained in this way 
are applicable to any open bounded set, not necessarily convex.
\par
Let $\Om$ be a bounded open set 
and denote by $\cK$ the closure of its convex hull; we shall prove the following inequality (see Theorem \ref{th:nonconvessi}):
\begin{equation}
\label{prova}
\dist(x_\infty,\pa\cK)\geq c_N\, \frac{\diam(\Om)}{[\diam(\Om)^{2}\,\la_1(\Om)]^N}\,.
\end{equation}
Here, $\diam(\Omega)$ is the diameter of $\Om$ and $c_N$ is a constant, depending only 
on the dimension $N$, for which we will give the precise expression; observe that the 
quantity $\diam(\Om)^{2}\,\la_1(\Om)$ is scale invariant.
\par
When $\Omega$ is convex, 
more explicit bounds can be derived; for instance, the following one relates the distance of $x_\infty$ from $\pa\Om$ to the inradius and the diameter:
\begin{equation}
\label{prova1}
\dist(x_\infty,\pa\Om)\geq C_N\, r_\Om \left(\frac{r_\Om}{\diam(\Om)}\right)^{N^2-1},
\end{equation}
where again $C_N$ is a constant depending only on $N$ (see Theorem \ref{th:conv} for its expression). 
We point out that the so called {\em Santal\`o point} of $\Omega$ always satisfies \eqref{prova1}, hence this can also be used to locate such a point (see Section \ref{sec:polar} and Remark \ref{santalo}).
\par
The latter argument relies instead on the following idea from \cite{Fr,GNN}.
Let $\SSN$ be the unit sphere in $\RN.$ For $\om\in\SSN$ and $\la\in\RE$
define the hyperplane 
\begin{equation}
\label{hyperplane}
\pi(\la,\om)=\{x\in\RN: x\cdot\om=\la\},
\end{equation}
and the two half-spaces
\begin{equation}
\label{halfspaces}
\pi^+(\la,\om)=\{x\in\RN: x\cdot\om>\la\}\quad\quad \pi^-(\la,\om)=\{x\in\RN: x\cdot\om\leq\la\}
\end{equation}
(here the symbol $\cdot$ denotes the usual scalar product in $\RE^N$). 
Suppose $\pi(\la,\om)$ has non-empty intersection with
the interior of the conductor $\Om$ and set 
$$
\Om^+_{\la,\om}=\Om\cap\pi^+(\la,\om).
$$ 
Then
{\it if the reflection $\cT_{\la,\om}(\Om^+_{\la,\om})$ of $\Om_{\la,\om}$ with respect to the hyperplane $\pi(\la,\om)$ lies in $\Om,$ then
$\pi(\la,\om)$ cannot contain any critical point of $u.$} This is a simple consequence of
{\it Alexandrov's reflection principle} based on Hopf's boundary point lemma
(see Section \ref{sec:alexandrov} for details).
\par
Based on this remark, for a convex body $\cK$ we can define a (convex) set 
$\core(\cK)$ --- the {\it heart} of $\cK$ ---
such that $x(t)\in\core(\cK)$ for every $t>0$ 
(in fact, we will prove that $\cK\setminus\core(\cK)$ cannot contain the hot spot $x(t)$ for any $t>0$).
The heart $\core(\cK)$ of $\cK$ is easily obtained as
the set
$$
\core(\cK)=\bigcap\,\{\pi^-(\la,\om) : \cT_{\la,\om}(\cK^+_{\la,\om})\subset\cK\}.
$$
\par
As we shall see, the two methods have their advantages and drawbacks, but they are, in a sense, complementary.
On the one hand, while inequalities \eqref{prova} and \eqref{prova1} 
are quite rough in the case in which $\Om$ has some simmetry
(e.g. they do not allow to precisely locate $x_\infty$ even when $\Omega$ is a ball), by the second argument,
the problem of locating $x(t)$ is quite trivial; on the other hand, while in some cases
(e.g. when $\Om$ has no simmetries or $\pa\Om$ contains 
some flat parts, as Example \ref{exa:joint} explains), we cannot exclude that
the heart of $\cK$ extends up to the boundary $\pa\cK$ of $\cK,$
estimates \eqref{prova} and \eqref{prova1} turn out to be useful to quantitatively bound
$x(t)$ away from $\pa\cK.$
Thus, we believe that a joint use of both of them provides a very useful method to locate $x(t)$ or $x_\infty$.
\par
Studies on the problem of locating $x_\infty$ can also be found in \cite{GJ}: there,
by arguments different from ours and for the two-dimensional case, the location of $x_\infty$ is estimated within a distance comparable to the inradius, uniformly for arbitrarily large diameter.
\par
In Section \ref{sec:folding} we shall relate $\core(\cK)$ to a function $\mathcal{R}_\cK$ of the direction $\om$
--- the {\it maximal folding function} --- and we will construct ways to characterize it. We will also connect $\cR_\cK$ to the Fourier transform of the characteristic function of $\cK$: this should have some interest from a numerical point of view.
  Finally, in Section \ref{sec:algoritmo},
we will present an algorithm to compute $\cR_\cK$ when $\cK$ is a polyhedron:
based on this
algorithm, we shall present some numerical computations.

\section{Hot spots and polar sets}
\label{sec:polar}

In this section, if not otherwise specified, $\Omega$ is a bounded open set and we denote by $\cK$ the closure of the convex hull of $\Omega$. Notice that $\cK$ is a convex body, that is a compact convex set, with non empty interior. In what follows, $|E|$  denotes
the $N$-dimensional Lebesgue measure of a set $E\subset\RE^N $
and $|\pa E|$ the $(N-1)$-dimensional Hausdorff measure of its boundary;
also, $\om_k$ will be the volume of the unit ball in $\RE^k.$

\subsection{Preliminaries}

We recall here some notations from \cite{Sc}.
 The {\it gauge function $j_{p}$ of $\cK$ centered at a point $p\in\cK$} is the function defined by
\[
j_{p}(x)=\min\{\la\ge 0\ :\ x-p\in\la(-p+\cK)\},\ x\in\RE^N\,.
\]
Observe that we have $j_p(t\, (x-p)+p)=t\, j_p(x)$ for every $t>0$; in particular, if $0\in\cK$ then $j_0$ is 1-homogeneous.
We set
\begin{equation}
\label{g1}
g_p(x)=\left\{\begin{array}{cl}
j_{p}\left(x\right)-1,& \mbox{ if } x\in\RE^N\setminus\{p\},\\
-1,& \mbox{ if } x=p.
\end{array}
\right.
\end{equation}
so that $g_p$ is the convex function whose graph is the cone projecting $\pa\cK$ from the point $(p,-1)\in\RE^{N+1}$.

It is also useful to recall the definition of the
{\em support function} $h_\cK$ of $\cK,$ that is
\begin{equation}
\label{suppfun}
h_{\cK}(\xi)=\max\{x\cdot \xi\,:\, x\in\cK\}\,,\quad \xi\in\RE^N.
\end{equation}
As it is easily seen, $h_\cK$ is a $1$-homogeneous convex function; viceversa, 
to any convex $1$-homogeneous function $h,$ it corresponds exactly one convex body whose
support function is $h$ (refer to \cite{Sc}, for instance). 
\par
The {\em polar set of $\cK$ with respect to $p$} 
is the convex set $\cK^*_{p}$ coinciding with the unit ball of the ``norm''\footnote{Properly speaking this is not a norm, since in general we have $\|-x\|*\not=\| x\|_*$.} $\|\cdot\|_*=h_\cK(\cdot)$ centered at $p$, that is
$$
\cK^*_p=\bigl\{y\in\RE^N\,:\, (x-p)\cdot (y-p)\leq 1\,\text{ for every }x\in\cK\bigr\};
$$
if $p$ is in the interior of $\cK$, then $\cK^*_p$ is compact.
Observe that this can be equivalently defined as
\[
\cK^*_p=\bigl\{y\in\RE^N\,:\, (x-p)\cdot (y-p)\leq j_p(x)\,\text{ for every }x\in\RE^N\bigr\}.
\]
We also recall that for every convex body $\cK$ the function $\psi:\cK\to[0,\infty]$ defined by
$\psi(x)=|\cK^*_x|$ attains a positive minimum at some point $s_\cK\in\cK$, which is  called the {\em Santal\`{o} point} of $\cK$ (see \cite{Sc}). 
When referring to the polar of $\cK$ with respect to its Santal\`{o} point, we simply write $\cK^*$, instead of $\cK^*_{x_\cK}$: we will see that the method developed in the next subsection for estimating the hot spot, applies to the Santal\`o point as well.
\par 
It is not difficult to see that $\cK^*_p$ coincides (up to a translation) with the {\it subdifferential} $\pa g_p$ of $g_p$ at the point $p$, i.e.
\begin{equation}
\label{polar}
\pa g_p(p)=\cK^*_p-p,
\end{equation}
where for  every $x_0$
$$
\pa g_p(x_0)=\{\xi\in\RE^N: g_p(x)\ge g_p(x_0)+\xi\cdot (x-x_0),\, x\in\RE^N\}.
$$
\par
Finally, we will need the following monotonicity property of the subdifferential of a function:
the proof can be found in \cite{Gu}.
\begin{lm}
\label{th:sub}
Let $u_1$ and $u_2$ be continuous convex functions on $\cK$ such that $u_1=u_2$ on $\pa\cK.$ 
Define:
\[
\pa u_i(\cK)=\bigcup\limits_{x\in \cK}\pa u_i(x), \ \ i=1, 2.
\]
If $u_1\le u_2$ in $\cK$, then
\(
\pa u_2(\cK)\subseteq \pa u_1(\cK).
\)
\end{lm}

\subsection{The polar set of $\cK$ with respect to the hot spot}

The following result holds for a general domain $\Om$ and is the cornerstone of our estimates. 

\begin{teo}\label{thmestimx0}
(i) Let $u$ be the solution of the initial-boundary value problem \eqref{heatconductor} and,
for every fixed time $t\in(0,+\infty)$, let $x(t)\in\Omega$ be a hot spot at time $t$,
that is a point where the value
$$
M(t)=\max\limits_{\overline\Omega} u(\cdot,t)
$$ 
is attained.
\par
Then
\begin{equation}\label{eqthm1}
|\cK^*_{x(t)}|\le [N\, M(t)]^{-N}\int_{\cC(t)} |u_t(x,t)|^N dx,
\end{equation}
where $\cC(t)$ is the {\em contact set} at time $t$, i.e. the subset of $\Om$ where $-u(\cdot,t)$ coincides with its convex envelope.
\par
(ii) Let $\la_1(\Om)$ and $\phi_1$ be respectively 
the first Dirichlet eigenvalue and eigenfunction of $-\Delta$ in $\Om.$ 
Let $x_\infty$ be  a maximum point of $\phi_1$ in $\Om$ and set $M_\infty=\phi_1(x_\infty).$
\par
Then
\begin{equation}
\label{eqthm2}
|\cK^*_{x_\infty}|\le \left[\frac{\lambda_1(\Omega)}{N\, M_\infty}\right]^N \int_{\cC} \phi_1(x)^N\ dx\,,
\end{equation}
where $\cC$ is the contact set of $\phi_1$, i.e. the subset of $\Om$ 
where $-\phi_1$ coincides with its convex envelope.
\end{teo}
\begin{proof}
(i) For $t\in(0,\infty),$ let $U^{(t)}:\cK\to\RE$ denote the convex envelope of $-u(\cdot,t);$ 
then, $U^{(t)}$ is a continuous convex function in $\cK$, such that $U^{(t)}=0$ on $\partial\cK.$
\par
The function $G$ defined by $G(x)=M(t)\, g_{x(t)}(x),$ for $x\in\cK,$
is such that
$$
G \ge U^{(t)} \ \mbox{ in } \cK \quad \mbox{ and } \quad G=U^{(t)} \quad \mbox{ on } \ \pa\cK,
$$
and hence
\(
\pa G(\cK)\subseteq \pa U^{(t)}(\cK)
\)
by Lemma \ref{th:sub}.
By the rescaling properties of the subdifferential and \eqref{polar}, we know that
$$
\pa G(\cK)=M(t)\, (\cK^*_{x(t)}-x(t)),
$$
thus, 
$$
|\cK^*_{x(t)}|\le M(t)^{-N}|\pa U^{(t)}(\cK)|\,.
$$
\par
On the other hand, by Sard's Lemma and the formula for change of variables (see for instance \cite[Section 1.4.2]{Gu}), we obtain
$$
|\pa U^{(t)}(\cK)|\leq\int_{\cC(t)}|\det D^2 u(x,t)|\ dx\,,
$$
with $\cC(t)=\{x\in\Omega\ :\ U^{(t)}(x)=-u(x,t)\}$. Observe that the contact set is not empty, thanks to the fact that $x(t)\in\cC(t)$ and moreover we have $|\cC(t)|>0$.
Now, by the arithmetic-geometric mean inequality, we have in $\cC(t)$ that
$$
|\det D^2 u(x,t)|^{1/N}\leq \frac{|\De u(x,t)|}{N}\,,
$$
which yields
$$
\int_{\cC(t)}|\det D^2 u(x,t)|\ dx\leq N^{-N}\int_{\cC(t)} |\De u(x,t)|^N\ dx.
$$
\par
Therefore, we finally obtain that
$$
|\cK^*_{x(t)}|\leq [N\, M(t)]^{-N}\int_{\cC(t)} |\De u(x,t)|^N\ dx\,.
$$
and we conclude the proof by simply using the equation $\De u=u_t.$
\par
(ii) The proof runs similarly to case (i).
\end{proof}

Estimates \eqref{eqthm1} and \eqref{eqthm2} are generally difficult to handle. The following weaker forms of \eqref{eqthm2} may be more useful.
\begin{coro}
\label{cor:estimxinfty}
Under the same assumptions as in Theorem \ref{thmestimx0}, we have:
\begin{equation}\label{coroll}
|\cK^*_{x_\infty}|\le \left[\frac{\lambda_1(\Omega)}{N}\right]^N |\Om|\,.
\end{equation}
and
\begin{equation}\label{coroll2}
|\cC|\ge\left[\frac{N}{\lambda_1(\Omega)}\right]^N\,|\cK^*|\, ;
\end{equation}
we recall that $\cK^*$ denotes the polar set of $\cK$ with respect to the Santal\`{o} point.
\end{coro}
\begin{oss}\label{rem:coroll}
As is well-known, $|\cK^*|$ can be estimated from below by 
$m_N /|\cK|,$ where $m_N$ is a positive constant (see \cite{Sc}).
Thus, \eqref{coroll2} becomes
$$
\frac{|\cC|}{|\cK|}\ge m_N\,\left[\frac{N}{\lambda_1(\Omega)|\cK|^{2/N}}\right]^N.
$$
\end{oss}
\begin{oss}\label{rem:GNN}
In \cite[p. 223]{GNN} the following problem is posed:  
\begin{quotation}
[...] Suppose $u>0$ is a solution of
\begin{equation}\label{GNN}
-\Delta u=f(u) \,\mbox{ in } \,\Omega, \quad u=0 \,\mbox{ on } \,\partial\Omega
\end{equation}
in a bounded domain $\Omega$ in $\mathbb{R}^N$, say $u\in
C^2(\overline{\Omega})$. Is there some $\varepsilon>0$ depending only on
$\Omega$ (i.e., independent of $f$
and $u$) such that $u$ has no stationary point in an
$\varepsilon$-neighbourhood of $\pa\Om$?
\par
This is true for $N=2$ in case $f(u)\ge 0$ for $u\ge 0$, but for $N> 2$ the
problem is open. [...]
\end{quotation}
\par
Here, we point out that by the same arguments used 
for Theorem \ref{thmestimx0}, we can easily prove the following estimate:
\begin{equation}
\label{est-for-f(u)}
|\cK_{x_0}^*|\le \left(\frac{|f(0)|+LM_0}{N M_0}\right)^N |\Om|,
\end{equation}
where $L$ is the Lipschitz constant for $f$ 
and $x_0$ is a point where $u$ achieves its maximum $M_0.$ 
When $f(0)=0,$ we obtain the inequality
\begin{equation}
\label{est-for-f(u)2}
|\cK_{x_0}^*|\le \left(\frac{L}{N}\right)^N |\Om|
\end{equation}
--- an estimate, similar to \eqref{coroll}, that can be used to bound $\dist(x_0,\pa\cK)$
from below in a way similar to that of Theorem \ref{th:nonconvessi} below. 
\par
An interesting instance of \eqref{est-for-f(u)} occurs when $f\equiv 1$ --- in this
case $u$ is the {\it torsional creep} of an infinite bar with cross-section $\Om;$ 
we thus obtain:
$$
|\cK_{x_0}^*|\le \frac{|\Om|}{(N M_0)^N}.
$$
This inequality can also be viewed as an estimate for the maximum $M_0$ in the spirit of the
Alexandrov-Bakelman-Pucci principle.
\end{oss}

Using the definition of the polar set, it is easy to see that $|\cK^*_x|$ goes to $\infty$ as the point $x$ approaches the boundary.
The following lemma gives a quantitative version of this fact and helps us to provide  
explicit estimates of the position of $x_\infty.$
\begin{lm}\label{lmestimK*}
Let $p$ be any point belonging to the interior of $\cK$ and define 
$R(p)=\max\{ |p-y|\,:\,y\in\pa\cK\}$. 
Then 
\begin{equation}
\label{boundK*}
|\cK^*_p|\geq\frac{\om_{N-1}/N}{R(p)^{N-1}\,\dist(p,\pa\cK)}\,.
\end{equation}
\end{lm}
\begin{proof}
Set $d=\dist(p,\pa\cK)$ and $R=R(p).$ 
Obviously $\cK$ is contained in the ball $B(p,R)$ centered at $p$ with radius $R$  and in the halfspace 
$$
H^+=\{y\in\RE^N\,:\,(y-\bar p)\cdot (\bar p-p)\leq d^2\}
$$
supporting $\cK$ at any point $\bar p$ such that
$|p-\bar p|=d.$
Set $E=B(p,R)\cap H^+;$ then 
$\cK\subseteq E$
whence
$$
\cK^*_p\supseteq E_p^*\,.
$$
Now notice that $E_p^*$ is the convex envelope of the union of the ball
$B(0,R^{-1})$ and the point $q=p+d^{-2}(\bar p-p)$; its volume is explicitly computed:
$$
|E_p^*|=\frac{\om_{N-1}}{N\,R^{N-1}\,d}\,
\left\{(1-\si^2)^{\frac{N+1}{2}}+N\,\si
\int_{-1}^\si (1-\tau^2)^{\frac{N-1}{2}}d\tau\right\},
$$
with $\si=d/R\in[0,1].$ 
Thus, \eqref{boundK*} is readily obtained by observing that the function of $\si$ into the braces 
is increasing and hence bounded below by $1.$
\end{proof}
We are now ready to prove the first quantitative estimate on the location of $x_\infty$: 
this will result from a combination of the previous lemma and \eqref{coroll}.
\begin{teo}
\label{th:nonconvessi}
Under the same assumptions of Theorem \ref{thmestimx0}, we have that
\begin{equation}\label{preciseestimd}
\dist(x_\infty,\partial\cK)\geq N^{N-1}\,\om_{N-1}\frac{\diam(\Omega)}
{\Big(|\Omega|^{1/N}\,\diam(\Omega)\,\lambda_1(\Omega)\Big)^N}\,\,.
\end{equation}
\par 
In particular, the following estimate holds true: 
\begin{equation}
\label{estimd}
\dist(x_\infty,\pa\cK)\geq \,2^N\,N^{N-1}\,\frac{\om_{N-1}}{\om_N} \frac{\diam(\Om)}{\Big(\diam(\Omega)^2\,\la_1(\Om)\Big)^N}.
\end{equation}
\end{teo}
\begin{proof}
Applying Lemma \ref{lmestimK*} with $p=x_\infty$ and Corollary \ref{cor:estimxinfty} gives:
$$
\left(\frac{\lambda_1(\Omega)}{N}\right)^N|\Om|\geq\frac{\om_{N-1}/N}{R(x_\infty)^{N-1}\,\dist(x_\infty,\pa\cK)}\,.
$$
Thus, \eqref{preciseestimd} easily follows by observing that $\diam(\cK)=\diam(\Omega)\geq R(x_\infty)$.
\par
Finally, using the isodiametric inequality 
$$
|\Omega|\leq\om_N\left[\frac{\diam(\Omega)}{2}\right]^N\!\!,
$$
in conjunction with \eqref{preciseestimd}, we show the validity of \eqref{estimd}.
\end{proof}

Estimates \eqref{preciseestimd} and \eqref{estimd} involve the first eigenvalue $\la_1(\Om)$, which in general is not easy to compute explicitly;
when $\Omega$ is convex, we can estimate $\lambda_1(\Omega)$ from above by means of basic geometric quantities, thus providing an easily computable lower bound on $\dist(x_\infty,\pa\Om)$. 
\par
This is the content of the following Theorem, which represents the main contribution of this section.
\begin{teo}
\label{th:conv}
If $\Omega$ is convex, then
\begin{equation}
\label{preciseestimd2}
\dist(x_\infty,\pa\Om)\ge r_\Om\,\left[\frac{\omega_{N-1} N^{2N-1}}{\la_1(B_1)^N}\,\, \mathrm{IPR}(\Om)^{-N}\,\left(\frac{r_\Omega}{\diam(\Om)}\right)^{N-1}\right],
\end{equation}
where $r_\Omega$ is the inradius of $\Om$, $\la_1(B_1)$ denotes 
the first Dirichlet eigenvalue of $-\De$ in the unit ball and $\mathrm{IPR}(\Om)=|\pa\Om||\Om|^{1/N-1}$ is the isoperimetric ratio of $\Om$.
\par
In particular, the following bound from below on $\dist(x_\infty,\pa\Om)$ holds true
\begin{equation}
\label{estimd2}
\dist(x_\infty,\pa\Om)\ge r_\Omega\left[\frac{(2^N N)^{N-1}}{\la_1(B_1)^N}\,\frac{\omega_{N-1}}{\omega_N}\,\left(\frac{r_\Omega}{\diam(\Om)}\right)^{N^2-1}\right].
\end{equation} 
\end{teo}
\begin{proof}
The proof of \eqref{preciseestimd2} readily follows by combining 
\eqref{preciseestimd} to the following upper bound on $\la_1(\Om)$
\begin{equation}\label{estimlambdar}
\lambda_1(\Omega)\leq\frac{\la_1(B_1)}{N}\,\frac{|\pa\Om|}{r_\Om|\Om|}\,,
\end{equation}
proved in \cite[Theorem 2]{FK}. Using 
\eqref{preciseestimd2} and the
two inequalities
$$
|\Om|\ge \om_N\, r_\Om^N \quad \mbox{ and } \quad |\pa\Om|\le N\,\om_N\,\left[\frac{\diam(\Om)}{2}\right]^{N-1},
$$ 
we end up with \eqref{estimd2}. 
\end{proof}
\begin{oss}
Observe that using 
\eqref{estimd} and the inequality
\[
\lambda_1(\Omega)\le \frac{\lambda_1(B_1)}{r_\Omega^2},
\]
which follows from the monotonicity and scaling properties of $\lambda_1$, we can infer
\[
\dist(x_\infty,\pa\Om)\ge r_\Omega\left[\frac{2^N N^{N-1}}{\la_1(B_1)^N}\,\frac{\omega_{N-1}}{\omega_N}\,\left(\frac{r_\Omega}{\diam(\Om)}\right)^{2N-1}\right],
\] 
thus providing a lower bound which is strictly greater than \eqref{estimd2}, as long as the ratio $r_\Omega/\mathrm{diam}(\Omega)$ is strictly smaller than $1/2$ and $N\ge 3$. 
\end{oss}
Inequality \eqref{estimlambdar} is in fact a corollary of a sharper inequality holding for starshaped sets. 
We can then give a refinement
of \eqref{preciseestimd2} which holds in this larger class:
to this end, we borrow some notations from \cite{FK}. 
\par
A set $\Om$ is said to be {\it strictly starshaped} with respect to a point $x_0\in\Om$ if it is starshaped with respect to $x_0$ and if its support function centered at $x_0$, i.e.
\[
h_{\Om,x_0}(x)=\max_{y\in \Om}\, (y-x_0)\cdot x,
\]
is uniformly positive, that is $\inf_{x\in\pa\Om} h_{\Om,x_0}(x)>0$. 
Let $\Omega$ be a strictly starshaped set with locally Lipschitz boundary, as in \cite{FK} we define 
\begin{equation*}
\label{defF}
W(\Om)=\inf\left\{\int_{\pa\Om} \frac{1}{h_{\Om,x_0}(x)}\,d\si(x)\,:\,x_0\in\Om\right\}\,,
\end{equation*}
where $d\si$ denotes surface measure on $\pa\Om.$
According to this notation, \cite[Theorem 3]{FK} states:
\begin{equation}\label{estimlambdaF}
\la_1(\Om)\leq\frac{\lambda_1(B_1)}{N}\,\frac{W(\Omega)}{|\Omega|}\,.
\end{equation}
Arguing as in Theorem \ref{th:conv} and using \eqref{estimlambdaF} in place of
\eqref{estimlambdar} gives the following estimate.
\begin{teo}
\label{th:ambiguo}
Let $\Omega$ be a strictly starhaped set with locally Lipschitz boundary and 
denote by $\cK$ the closure of the convex hull of $\Omega$. Then
\begin{equation}\label{preciseestimstar}
\dist(x_\infty,\pa\cK)\geq \frac{N^{2N-1}\,\om_{N-1}}{\la_1(B_1)^N}\,
\left(\frac{|\Omega|}{\diam(\Om)\, W(\Om)}\right)^{1-N}\,\frac{1}{W(\Om)}\,.
\end{equation}
\end{teo}
\begin{oss}
We remark that \eqref{preciseestimstar} is sharper and more general
than \eqref{preciseestimd2}, and it is at the same time more explicit than \eqref{preciseestimd}, in the sense that,
differently from $\la_1(\Om),$  the number $W(\Om)$  
can be computed directly from the support function (which exactly determines a convex set). 
\end{oss}

\begin{oss}\label{santalo}
It is worth noticing that the Santal\`{o} point $s_{\cK}$ of $\cK$ always satisfies
\eqref{coroll} (as well as \eqref{eqthm1} for every $t>0$), then it satisfies all the estimates we proved for $x_\infty$ in this section. In particular, Theorem \ref{th:conv} (or Theorem \ref{th:ambiguo}) can be used as well to  
estimate the location of the Santal\`o point of a convex set.
\end{oss}

\setcounter{teo}{0}
\setcounter{equation}{0}

\section{Alexandrov's reflection principle}
\label{sec:alexandrov}

In this section, for the reader's convenience, we recall some relevant facts
about {\it Aleksandrov's symmetry principle,} which has been extensively used
in many situations and with various generalizations (see \cite{Fr} for a good reference).

For $\om\in\SSN,$ let $\pi(\la,\om),$ $\pi^+(\la,\om),$ and $\pi^-(\la,\om)$
be the sets defined in \eqref{hyperplane} and \eqref{halfspaces}.
Also, define a linear transformation $A_\om:\RN\to\RN$
by the matrix:
\[
\cA_\omega=(\delta_{ij}-2\omega_i\omega_j)_{i,j=1,\dots,N}
\]
where $\delta_{ij}$ is the Kronecker symbol and the $\omega_i$ are the components of $\omega$.
Then the application $\cT_{\la,\om}:\RN\to\RN$ defined by
\[
\cT_{\la,\om}(x)=\cA_\om\, x+2\la\om,\ x\in\RE^N,
\]
represents the reflection with respect to $\pi(\la,\om).$
As already mentioned, if $\Om$ is a subset of $\RN,$  we set $\Om^+_{\la,\om}=\Om\cap\pi^+(\la,\om).$

\begin{prop}
\label{th:nohotspot}
Let $\Om$ be a bounded domain in $\RN$ with Lipschitz continuous boundary $\pa\Om$ and
suppose the hyperplane $\pi(\la,\om)$ defined by
 \eqref{hyperplane} has non-empty
intersection with $\Om.$ Assume that $\cT_{\la,\om}(\Om^+_{\la,\om})\subset\Om.$
\par
If $\Om$ is not symmetric with respect to $\pi(\la,\om)$, then $\pi(\la,\om)$ does not contain any (spatial) critical point of the
solution $u$ of \eqref{heatconductor}.
\end{prop}
\begin{proof}
For $x\in\Om^+_{\la,\om}$ and $t>0$ the function
$$
v(x,t)=u(\cT_{\la,\om}\, x,t)-u(x,t)
$$
is well-defined and is such that
\begin{equation}
\label{propertyv}
\begin{array}{ll}
v_t=\Delta v\quad\mbox{ in }\ &\Om^+_{\la,\om}\times (0,\infty),\\
v=0 \quad\quad\mbox{ on } &\Om^+_{\la,\om}\times\{ 0\},\\
v\ge 0 \quad\quad\mbox{ on } &\pa\Om^+_{\la,\om}\times
(0,\infty).
\end{array}
\end{equation}
Hence $v>0$ in $\Om^+_{\la,\om}\times (0,\infty),$ by the strong maximum principle for parabolic operators (see \cite{PW}).
Since $v=0$ on $(\pa\Om^+_{\la,\om}\cap\pi(\la,\om))\times(0,\infty),$
we obtain that $\frac{\pa v}{\pa\om}>0$ on it ($\om$ is in fact the interior normal unit vector),
by Hopf's boundary lemma for parabolic operators. We conclude
by noticing that $\ds\frac{\pa v}{\pa\om}=-2\,\frac{\pa u}{\pa\om}$ on $(\pa\Om^+_{\la,\om}\cap\pi(\la,\om))\times(0,\infty).$
\end{proof}
\par
With the same arguments and a little more work, one can extend this result 
to more general situations, involving nonlinearities both
for elliptic and parabolic operators. As an example, 
here we present the following result.
\begin{prop}
\label{th:nocritical}
Let $\Om$ and $\Om^+_{\la,\om}$ satisfy the same assumptions as those of Proposition
\ref{th:nohotspot}; in particular suppose that $\cT_{\la,\om}(\Om^+_{\la,\om})\subset\Om.$
\par
Let $u=u(x)$ be a solution of class $C^1(\ovr{\Om})\cap C^2(\Om)$ of the system:
\begin{equation}
\label{elliptic}
\begin{array}{c}
\De u+f(u)=0 \ \mbox{ and } \ u>0 \ \mbox{ in }\ \Om,\\
u=0 \quad\quad\mbox{ on } \pa\Om,
\end{array}
\end{equation}
where $f$ is a locally Lipschitz continuous function.
\par
If $\Om$ is not symmetric with respect to $\pi(\la,\om)$, then $\pi(\la,\om)$ does not contain any critical point of $u.$
\end{prop}
\begin{proof}
The proof runs similarly to that of Proposition \ref{th:nohotspot}; the relevant changes follow.
The function
$$
v(x)=u(\cT_{\la,\om}\, x)-u(x),
$$
defined for $x\in\Om_{\la,\om},$ satisfies the conditions:
\begin{equation*}
\label{propertyvell}
\begin{array}{lll}
&\De v+c(x)\, v=0 &\mbox{ in }\ \Om^+_{\la,\om},\\
&v\ge 0 &\mbox{ on } \pa\Om^+_{\la,\om},
\end{array}
\end{equation*}
where the function $c(x),$ defined by
$$
c(x)=\left\{\begin{array}{ll}
\ds\frac{f(u(\cT_{\la,\om}\, x))-f(u(x))}{u(\cT_{\la,\om}\, x)-u(x)} &\mbox{ for } u(\cT_{\la,\om}\, x)\not=u(x),\\
\ds 0 &\mbox{ for } u(\cT_{\la,\om}\, x)=u(x),
\end{array}\right.
$$
is bounded by the Lipschitz constant of $f$ in the interval $[0,\max_{\ovr{\Om}}u].$
Hence $v\ge 0$ in $\Om^+_{\la,\om},$ by the arguments used in \cite{Fr}.
Let $c^-(x)=\max(-c(x),0);$ then
\begin{equation*}
\De v-c^-(x)\, v\le 0 \ \mbox{ and } \ v\ge 0 \ \mbox{ in }\ \Om^+_{\la,\om}
\end{equation*}
and the strong maximum principle can be applied to obtain that $v>0$ in $\Om^+_{\la,\om}.$
The conclusion then follows as before
by Hopf's boundary lemma.
\end{proof}
\par
An immediate consequence of this theorem is the following result.
\begin{coro}
\label{coro:nohotspot}
Let $\Om$ and $\Om^+_{\la,\om}$ satisfy the same assumptions as those of Proposition
\ref{th:nohotspot}.
Let $u_1$ be the first (positive) eigenfunction of $-\De$ with homogeneous Dirichlet
boundary conditions.
\par
If $\Om$ is not symmetric with respect to $\pi(\la,\om)$, then $\pi(\la,\om)$ does not contain any critical point of $u_1.$
\end{coro}
\setcounter{teo}{0}
\setcounter{equation}{0}

\section{The heart of a convex body}
\label{sec:folding}

In what follows, we shall assume that $\cK\subset\RE^N$ is a convex body, that is a compact convex set with non-empty interior.
Occasionally, we will suppose that $\cK\subset\RE^N$ is of class $C^1,$ i.e. a set whose boundary $\pa\cK$ is an $(N-1)$-dimensional submanifold
of $\RE^N$ of class $C^1.$

\subsection{The maximal folding function}

We are interested in determining the function given by
\begin{equation}
\label{lambda}
\cR_\cK(\om):=\min\{\la\in\RE\ :\ \cT_{\la,\om}(\cK^+_{\la,\om})\subseteq\cK\},\ \ \om\in \SSN,
\end{equation}
which will be called the {\it maximal folding function} of $\cK;$
$\cR_\cK$  defines in turn a subset of $\cK$ -- the {\it heart} of $\cK$ -- as
$$
\core(\cK)=\{x\in\cK\ :\ x\cdot\omega\le\cR_\cK(\omega), \mbox{ for every } \omega\in \SSN\}.
$$
Of course, $\core(\cK)$ is a closed convex subset of $\cK.$
Observe that $\cR_\cK$ can be bounded below and above by 
means of the support functions of $\core(\cK)$ and $\cK:$
\begin{equation}
\label{supp_core}
h_{\core(\cK)}(\om)\le\cR_\cK(\om)\le h_\cK,\ \om\in \SSN\,.
\end{equation}
\par
The following results motivate our interest on $\core(\cK)$ and $\cR_\cK.$
\begin{prop}
\label{prop:heart}
Let $\cK$ be a convex body. 
\begin{itemize}
\item[(i)] The hot spot $x(t)$ of $\cK,$ the point $x_\infty$ and any limit
point of $x(t)$ as $t\to 0^+$ always belong to $\core(\cK);$ 
moreover, $x(t)$ and $x_\infty$ must fall in the interior of $\core(\cK),$
whenever this is non-empty.
\item[(ii)] The center of mass of $\cK,$
$$
\ovr{x}_\cK=\frac1{|\cK|}\int_\cK y\, dy\,,
$$
always belongs to the heart $\core(\cK)$ of $\cK.$ 
\item[(iii)]
If $\cK$ is strictly convex, the incenter $x_\cK^I$ of $\cK$ belongs to $\core(\cK).$  
\item[(iv)]
Let $\ovr{x}_\cK=0.$ If there exist $\ell$ $(1\le\ell\le N)$ independent directions $\om_1,\dots,\om_\ell$ such that
$\cR_\cK(\om_j)=0,$ $j=1,\dots, \ell,$ then
$$
\core(\cK)\subset\cK\cap\bigcap_{j=1}^\ell\pi(0,\om_j).
$$
In particular, if $\ell=N,$ then $\core(\cK)$ reduces to $\ovr{x}_\cK$ and the hot spot of $\cK$ is stationary.
\item[(v)] Let
\begin{equation}
\label{R_K}
\mathfrak{r}_\cK=\max_{\theta\in\SSN}\left\{\min_{\om\cdot\theta>0 }\frac{\cR_\cK(\om)-\ovr{x}_\cK\cdot\om}{\theta\cdot\om}\right\},
\end{equation}
then
$$
\core(\cK)\subseteq B(\ovr{x}_\cK, \mathfrak{r}_\cK).
$$
\end{itemize}
\end{prop}
\begin{proof}
Items (i) and (iv) follow by observing that, for $\la=\cR_\cK(\om),$ the
set $\cK^+_{\la,\om}\cup\cT_{\la,\om}(\cK^+_{\la,\om})$ is contained in $\cK$
and is symmetric with respect to $\pi(\la,\om).$ Hence,
\begin{eqnarray*}
\la-\ovr{x}_\cK\cdot\om=
\frac{1}{|\cK|}\int_{\cK\setminus(\cK^+_{\la,\om}\cup\cT_{\la,\om}(\cK^+_{\la,\om}))} [\la-y\cdot\om]\ dy
\end{eqnarray*}
and the last term is non-negative, vanishing if and only if $\cK$ is $\om$-symmetric.
\par
Items (ii) and (iii) are easy consequences of Proposition \ref{th:nohotspot} and Corollary \ref{coro:nohotspot}.
\par
For a fixed $\theta\in\SSN,$ let us define
\[
\alpha(\theta)=\max\{t\, :\, \ovr x_\cK+t\theta\in\core(\cK)\},
\] 
which is non negative, thanks to (ii).
Then $x=\ovr{x}_\cK+\al(\theta)\theta\in\core(\cK)$ and 
$$
\ovr{x}_\cK\cdot\om+\al\,\theta\cdot\om\le\cR_\cK(\om),
$$
for every $\om\in\SSN$ such that $\omega\cdot\theta>0$. Hence
$$
\al\le\min_{\om\in\SSN}\frac{\cR_\cK(\om)-\ovr{x}_\cK\cdot\om}{\theta\cdot\om},
$$
thus taking the maximum as $\theta$ varies on $\SSN$ we obtain \eqref{R_K}.
\end{proof}
Informations
 on convex heat conductors with a stationary hot spot can be found in \cite{CS,GW,Ka,Kl,MS1,MS2}.

\begin{oss}
Formula \eqref{R_K} deserves some comments: observe that for every fixed $\theta\in\SSN$, the minimum problem inside the braces amounts to finding a  direction $\om$ close to $\theta$, so to maximize $\theta\cdot\omega$, and such that at the same time we can fold $\cK$ as much as possible, so to minimize the difference $\cR_\cK(\omega)-\ovr x_\cK\cdot\omega$. 
\end{oss}
\par
We conclude this subsection by an example that shows how the simultaneous
application of Proposition \ref{prop:heart} and the results of Section 2 
substantially benefits the problem of locating $x_\infty.$  
\begin{exa}
\label{exa:joint}
Let us consider a spherical cap
\[
B^+_\mu=\left\{(x_1,\dots,x_N)\in\RE^N\, :\, \sum_{i=1}^N x_i^2 =R^2,\, x_N\ge \mu\right\},
\]
with $0\le \mu<R$.
Thanks to the simmetry of $B^+_\mu$, it is easily seen that its heart is given by
\[
\core(B^+_\mu)=\{(x_1,\dots,x_N)\in\RE^N\, :\, x_1=\dots=x_{N-1}=0,\, \mu\le x_N\le (R+\mu)/2\},
\]
which is a vertical segment touching the boundary $\pa B^+_\mu$ at the point $(0,\dots,0,\mu)$. In particular, by this method we can not exclude that the hotspot $x(t)$ (or the point $x_\infty$) is on the boundary. However, we can now use the results of Section \ref{sec:polar}, to further sharpen this estimate on the location of $x_\infty$: indeed, applying Theorem \ref{th:conv}, we get
\[
\mathrm{dist}(x_\infty,\partial B^+_\mu)\ge (R-\mu)\,
\left[\frac{2^{1-N-N^2}N^{N-1}}{\lambda_1(B_1)^N}\,\frac{\omega_{N-1}}{\omega_N}\,\left(\frac{R-\mu}{R+\mu}\right)^{(N^2-1)/2}\right],
\]
where we used that $\diam(B^+_\mu)=2\sqrt{R^2-\mu^2}$ and $r_{B^+_\mu}=(R-\mu)/2$. 
\end{exa}

\subsection{Computing $\cR_\cK$}

The following theorem whose proof can be found in \cite[Theorem 5.7]{Fr} guarantees that, for a regular set (not necessarily convex),
the maximal folding function is never trivial.
\begin{teo}
\label{th:fraenkel}
Let $\Om$ be a bounded open (not necessarily convex) subset of $\RE^N,$ with $C^1$ boundary $\pa\Om$, and
denote by $\cK$ the convex hull of $\Omega.$ 
\par
For every $\om\in \SSN$,
there exists $\ep>0$ such that,
for every $\la$ in the interval $(h_\cK(\omega)-\ep,h_\cK(\om)),$ we have:
\begin{enumerate}
\item[(i)] $\cT_{\la,\om}(\Om^+_{\la,\om})\subset\Om$;
\item[(ii)] $\nu(x)\cdot\omega>0$, for every $x\in\pa\Om\cap\pi^+(\la,\om)$.
\end{enumerate}
\end{teo}
Unfortunately, the previous result is just qualitative and does not give any quantitative information about the maximal folding function.
Moreover, notice that the $C^1$ assumption on $\pa\Om$ cannot be dropped, even in the case of a convex domain: think of the spherical cap in Example \ref{exa:joint}, for which we have $\cR_{B^+_\mu}(-e_N)=h_{B^+_\mu}(-e_N)$.
\vskip.3cm
In order to compute $\cR_\cK,$ we need some more definitions.
We set
$$
\omp=\pi(0,\om)
$$ 
and for every $y\in\omp$ we define the segment
\[
\si_\om(y)=\{x\in\cK\ :\, x=y+t\om,\, t\in\RR\}.
\]
Then, we denote by $\cP_\om:\RE^N\to\om^\bot$ the projection operator on $\om^\bot$, that is the application defined by
\[
\cP_\om(x)=x-(x\cdot\om)\,\om,\ \ x\in\RE^N\,,
\]
and, for $y$ in the set
$$
\cS_\om(\cK)=\omp\cap\cP_\om(\cK)
$$
-- the {\it shadow of $\cK$ in the direction $\om$} -- we define:
$$
a_\om(y)=\min\{ t\in\RR: y+t\om\in\cK\} \ \mbox{ and } \ b_\om(y)=\max\{ t\in\RR: y+t\om\in\cK\}.
$$
\par
We say that a convex body $\cK$ is {\it $\om$-strictly convex} if
$\pa\cK$ does not contain any segment parallel to $\om.$
If $\cK$ is $\om$-strictly convex, then for every $x=y+t\om\in\pa\cK$ (equivalently $y\in\pa S_\om(\cK)$) 
such that the normal $\nu(x)$ to $\pa\cK$ at $x$ is
orthogonal to $\om$, the set $\si_\om(x)$ degenerates
to the singleton $\{x\}.$
\begin{oss}
We point out that $a_\om$ is a convex function on $\shk$, while $b_\om$ is concave;
moreover, if we set
\begin{eqnarray*}
&\mathrm{graph^+}(a_\om)=\{(y,t\,\om): y\in\shk,\ t\ge a_\om(y)\}, \\
&\mathrm{graph^-}(b_\om)=\{(y,t\,\om): y\in\shk,\ t\le b_\om(y)\},
\end{eqnarray*}
we have that
\[
\mathrm{graph^+}(a_\om)\cap\mathrm{graph^-}(b_\om)=\cK
\]
and, as soon as $\cK$ is $\om$-strictly convex,
\[
\mathrm{graph}(a_\om)\cup\mathrm{graph}(b_\om)=\pa\cK,
\]
where obviously $\mathrm{graph}(\cdot)$ denotes the graph of the relevant functions.
\end{oss}
\par
\begin{teo}
\label{th:brasco}
Let $\cK\subset\RE^N$ be a convex body. For $\om\in\SSN$ consider
the function $f:\shk\to\RE$ given by
\begin{equation}
\label{f}
f_\om(y)=\frac{a_\om(y)+b_\om(y)}{2}, \ \ y\in\shk.
\end{equation}
\par
Then
\begin{equation}
\label{lambda2}
\cR_\cK(\omega)=\max\limits_{y\in\shk} f_\om(y).
\end{equation}
\end{teo}
\begin{proof}
Observe that
$$
\cK^+_{\la,\om}=\{y+t\om : y\in\shk, \la<t<b_\om(y)\}.
$$
\par
Let $\ovr{\la}=\cR_\cK(\om);$ since $\cT_{\ovr{\la}, \om}(\cK^+_{\ovr{\la}, \om})\subset\cK,$
then, for every point $y+t \om$ with $y\in \shk$ and $\la<t<b_\om(y),$
we have that $\cT_{\ovr{\la}, \om}(y+t \om)\in\cK;$ in particular, for $t=b_\om(y),$
we obtain that $2 \ovr{\la}-b_\om(y)\ge a_\om(y)$ and hence $\ovr{\la}\ge f_\om(y).$
Thus,
$$
\cR_\cK(\om)\ge \max\limits_{y\in\shk} f_\om(y).
$$
\par
If $y_0\in\shk$ maximizes $f_\om,$ by taking $\la=f_\om(y_0),$
we see that $\cT_{\ovr{\la}, \om}(y+t \om)\in\cK$ for every $y\in \shk$ and $\la<t<b_\om(y).$
Therefore, $\cT_{\ovr{\la}, \om}(\cK^+_{\ovr{\la}, \om})\subset\cK$ and hence $\cR_\cK(\om)\le f_\om(y_0).$
\end{proof}
If we now remember that, for a convex domain $\cK$, the quantity
$$
w_\cK(\om)=h_\cK(\om)+h_\cK(-\om),
$$
is the {\it width of $\cK$ in the direction $\om$}, we immediately get a nice consequence of the previous Theorem.
\begin{coro}
\label{th:corewidth}
Let $\cK\subset\RE^N$ be a convex body. Then we have the following
estimate for the width of $\core(\cK)$ in the direction $\om$:
\begin{equation}
\label{corewidth}
w_{\core(\cK)}(\om)\le\osc\limits_{y\in\shk}f_\om(y).
\end{equation}
\end{coro}
\begin{proof}
We first observe that $\cS_{-\om}(\cK)=\shk$, so that
$$
f_{-\om}(y)=-f_\om(y),\ \ y\in\shk,
$$
and \eqref{lambda} yields
\begin{equation}
\label{-lambda}
\cR_\cK(-\omega)=-\min\limits_{y\in\shk} f_\om(y).
\end{equation}
Then, from the definition
of width, using (\ref{supp_core}), (\ref{lambda}) and (\ref{-lambda}), we get
\[
\begin{split}
w_{\core(\cK)}(\om)&=h_{\core(\cK)}(\om)+h_{\core{\cK}}(-\om)\\
&\le\cR_\cK(\om)+\cR_\cK(-\om)\\
&=\max\limits_{y\in\shk}f_\om(y)-\min\limits_{y\in\shk}f_\om(y)=\osc\limits_{y\in\shk}f_\om(y),
\end{split}
\]
thus concluding the proof.
\end{proof}
\begin{exa}
In general, inequality \eqref{corewidth} is strict. For example, in $\RE^2$ consider the ellipse given by
\[
\cK=\left\{(t,s)\in\RE^2\ :\ \frac{t^2}{a^2}+\frac{s^2}{b^2}=1\right\},
\]
with $0<b\le a$. The function $\cR_\cK$ can be easily computed in this case:
for every $\om=(\om_1,\om_2)\in\mathbb{S}^1$ we get
\[
\cR_{\cK}(\om)=\frac{a^2-b^2}{\sqrt{b^2\om_1^2+a^2\om_2^2}}\,|\om_1\om_2|
\]
--- the set $\{\cR_\cK(\om)\,\om: \om\in\mathbb{S}^1\}$ is the image of a  {\it quadrifolium}
(a {\it rhodonea} with $4$ petals) by the mapping $(x,y)\mapsto (x/a,y/b)$.
\par
Thus, for example, by choosing the direction $\omega=(1/\sqrt{2},1/\sqrt{2}),$ the right-hand side of \eqref{corewidth}
equals
\[
\frac{a^2-b^2}{2\sqrt{2}}\sqrt{\frac{1}{a^2+b^2}},
\]
while clearly the left-hand side is zero since $\core(\cK)=\{(0,0)\}$, due to the simmetries of $\cK$.

\par
This example also highlights the interest of the quantity $\osc_{\shk}f_\om-w_{\core(\cK)}$, which can be seen as a measure of the lack of symmetry of $\cK$ in
the direction of $\om$.
\end{exa}
The function $f_\om$ in \eqref{f} can be explicitly computed by the use of the Fourier transform: this is the content of the next result.
\begin{teo}
\label{th:fourier}
Let $\cK$ be a convex body and for $\om\in\SSN$, let $f_\om$ be the function defined in \eqref{f}. 
\par
\begin{equation}
\label{fourierf}
f_\om(y)=\frac{i\int_{\omp}\pa_\om\wh\one_\cK(\eta)\, e^{iy\cdot\eta}d\eta}{\int_{\omp}\wh\one_\cK(\eta)\, e^{iy\cdot\eta}d\eta},
\ y\in\shk,
\end{equation}
where $\wh\one_\cK$ denotes the Fourier transform of the characteristic function of $\cK$ and $\pa_\om$
differentiation in the direction $\om.$
\end{teo}
\begin{proof}
For $x\in\cK$ and $\xi\in\RN$ we write $x=y+t\om$ and $\xi=\eta+\tau\om,$ with $y\in\shk$, $\eta\in\omp$ and $t, \tau\in\RR$.
By Fubini's theorem we compute
\begin{eqnarray}
\label{charK}
\wh\one_\cK(\xi)&=&\int_\cK e^{-ix\cdot\xi}\ dx\nonumber\\
&=&\int_{\shk}\left(\int_{-\infty}^\infty \one_\cK(y+t\om)\, e^{-it\tau}dt \right) e^{-iy\cdot\eta}dy\nonumber\\
&=&\int_{\shk}\left(\int_{a_\om(y)}^{b_\om(y)} e^{-it\tau}dt \right) e^{-iy\cdot\eta}dy.
\end{eqnarray}
For $\tau=0$ we then obtain:
\begin{eqnarray}
\label{charK2}
\wh\one_\cK(\eta)=\int_{\shk}[b_\om(y)- a_\om(y)]\, e^{-iy\cdot\eta}dy.
\end{eqnarray}
Therefore, by the inversion formula for the Fourier transform, we have:
\begin{equation}
\label{diffab}
\frac1{(2\pi)^{N-1}} \int_{\omp}\wh\one_\cK(\eta)\, e^{iy\cdot\eta}d\eta=\left\{\begin{array}{cc}
b_\om(y)-a_\om(y),& y\in\shk,\\
0,& x\in\omp\setminus\shk.
\end{array}
\right.
\end{equation}
\par
By \eqref{charK}, we also obtain that
$$
\pa_\om\wh\one_\cK(\xi)=\frac{d}{d\tau}\wh\one_\cK(\eta+\tau\om)=
-i\,\int_{\shk}\left(\int_{a_\om(y)}^{b_\om(y)} t\, e^{-it\tau}dt \right) e^{-iy\cdot\eta}dy,
$$
$$ 
\pa_\om\wh\one_\cK(\eta)=-i\,\int_{\shk}\frac{b_\om(y)^2- a_\om(y)^2}{2}\, e^{-iy\cdot\eta}dy, \ \eta\in\omp,
$$
and hence
\begin{equation}        {}
\label{diffquadab}
\frac{i}{(2\pi)^{N-1}} \int_{\omp}\pa_\om\wh\one_\cK(\eta)\, e^{iy\cdot\eta}d\eta=\left\{\begin{array}{cc}
\frac{b_\om(y)^2-a_\om(y)^2}{2}, & y\in\shk,\\
0,& y\in\omp\setminus\shk.
\end{array}
\right.
\end{equation}
\par
Formula \eqref{fourierf} follows from \eqref{diffquadab} and  \eqref{diffab} at once.
\end{proof}

\begin{oss}
If $\cK$ is a polygon, $\wh\one_\cK$ can be explicitly computed in terms of the vertices of $\cK$. Let $\cK\subset\RR^2$ be a (convex) polygon with vertices
$p_1,\dots, p_n;$ we assume that $p_1,\dots, p_n$ are ordered counterclockwise and we set $p_{n+1}=p_1.$
\par
Rewriting $\wh\one_\cK$ as a boundary integral (see \cite{Herz}) by means of the divergence theorem, we have that  
\[
\wh\one_\cK(\xi)=-\frac{1}{|\xi|^2}\,\sums_{j=1}^n |p_{j+1}-p_j|(\nu_j\cdot\xi)\,\frac{e^{-ip_{j+1}\cdot\xi}-e^{-ip_j\cdot\xi}}{(p_{j+1}-p_j)\cdot\xi},
\]
where 
$$
\nu_j=
\left(
\begin{array}{cc}
0 &1\\
-1 &0
\end{array}
\right)
\frac{p_{j+1}-p_j}{|p_{j+1}-p_j|}, \ j=1,\dots, n,
$$
is the exterior normal to the $j$-th side of $\cK.$
Also, $\pa_\om\wh\one_\cK(\eta)$ is easily computed from the previous expression:
\[
\begin{split}
\ds\pa_\om\wh\one_\cK(\eta)&=\frac{1}{|\eta|}\,\sums_{j=1}^n |p_{j+1}-p_j|^2
\frac{e^{-ip_{j+1}\cdot\eta}-e^{-ip_j\cdot\eta}}{[(p_{j+1}-p_j)\cdot\eta]^2}\\
&+
\ds\frac{i}{|\eta|^2}\,\sums_{j=1}^n |p_{j+1}-p_j|(\nu_j\cdot\eta)\,
\frac{(p_{j+1}\cdot\om)\,e^{-ip_{j+1}\cdot\eta}-(p_j\cdot\om)\,e^{-ip_j\cdot\eta}}{(p_{j+1}-p_j)\cdot\eta}.
\end{split}
\]
\end{oss}

\subsection{Necessary optimality conditions}

We conclude this section by presenting some necessary conditions for the optimality
of $f_\om.$
To this aim, we first state and prove an easy technical result
for the subdifferential of a function.

\begin{lm}
\label{th:tech}
Let $\Om\subset\RE^k$ be a convex open set.
Let $\varphi$ and $\psi$ be a convex and, respectively, a concave function from $\Om$ to $\RE$. If
$\varphi+\psi$ attains its maximum at a point $y_0\in\Om$, then
\begin{equation}
\label{contenuto}
\pa\varphi(y_0)\subset \pa(-\psi)(y_0).
\end{equation}
\end{lm}
\begin{proof}
It is clear that both $\pa\varphi(y_0)$ and $\pa(-\psi)(y_0)$ are non-empty.
Since $y_0$ is a maximum point, we get
\[
\varphi(y_0)+\psi(y_0)\ge \varphi(y)+\psi(y)  \mbox{ for every } y\in\Om,
\]
and hence
\begin{equation*}
\label{sub}
\varphi(y)-\varphi(y_0)+\xi\cdot(y-y_0)\le -\psi(y)+\psi(y_0)+\xi\cdot(y-y_0) 
\end{equation*}
for every $\xi\in\RE^N,$ and $y\in\Om.$
If $\xi\in\pa\varphi(y_0),$ then we have $\xi\in\pa(-\psi)(y_0).$
\end{proof}

As a consequence of the definitions of $a_\om$ and $b_\om$, we have that
$\pa a_\om(y_0)\cup\pa(-b)_\om(y_0)=\varnothing$
implies that $y_0$ belongs to the boundary of $\shk.$

We are now in a position to state a necessary optimality condition.
\begin{teo}
\label{th:necessary}
Let $\cK\subset\RE^N$ be a convex body and $\om\in \SSN$. Suppose that $f_\om$
attains its maximum at a point $y_0\in\shk,$ that is
\[
\cR_\cK(\om)=f_\om(y_0).
\]
Set $\la=\cR_\cK(\om)$ and for every $x_0\in\cP^{-1}_\om(y_0)\cap\pa\cK$
denote by  $x_0^\la$ its reflection with respect to the optimal hyperplane, that is $x_0^\la=\cT_{\la,\om} x_0.$

Then: \begin{enumerate}
\item[(i)] if $x_0\not=x_0^\la$, we have
\begin{equation}
\label{necessary}
\cA_{\om}(N_\cK(x_0^\la))\subseteq N_\cK(x_0),
\end{equation}
where 
$$
N_\cK(x)=\{\xi\in\RN\setminus\{0\}: x\cdot\xi=h_\cK(\xi)\}\cup\{0\}.
$$
denotes the {\it normal cone} of $\cK$ at a point $x\in\pa\cK$;
\item[(ii)] if $x_0=x_0^\la$, there holds
\begin{equation}
\label{necessary_b}
\cA_{\om}(N^-_\cK(x_0)\})\subseteq N^+_\cK(x_0),
\end{equation}
where
\[
N^-_\cK(x)=\{\xi\in N_\cK(x)\, :\, \xi\cdot\omega\le 0\}\ \mbox{ and }\ N^+_\cK(x)=\{\xi\in N_\cK(x)\, :\, \xi\cdot\omega\ge 0\}.
\]
\end{enumerate}
\end{teo}
\begin{proof}
We can suppose for simplicity that $\om=e_N=(0, \dots, 0, 1).$ We first suppose that $y_0$ is
an interior point.
Since by its very definition $f_\om$ is the sum of a convex function and a concave one,
by Lemma \ref{th:tech} 
\begin{equation}\label{inclusionsubg} 
\pa a_\om(y_0)\subset\pa(-b_\om)(y_0)\,.
\end{equation}
Let us now set $x_0=y_0+b_\om(y_0)\,\om;$ the reflection of $x_0$
in the optimal hyperplane is $x_0^\la=y_0+a_\om(y_0)\,\om$ and
$$
N_\cK(x_0)=\{(\eta,1)\,:\,\eta\in\pa(- b_\om)(y_0)\}\,\quad
N_\cK(x_0^\la)=\{(\eta,-1)\,:\,\eta\in\pa a_\om(y_0)\}\,,
$$
for
$
\mathrm{graph^+}(a_\om)\cap\mathrm{graph^-}(b_\om)=\cK.
$
Then, since $\cA_\om(\eta,1)=(\eta,-1)$, \eqref{inclusionsubg} implies \eqref{necessary}.
\vskip.3cm
If $y_0$ is on the boundary of $\shk,$ then $\pa(- b_\om)(y_0)$
or $\pa a_\om(y_0)$ may be empty: this is clearly the case if $\mathrm{graph}(a_\om)$ or $\mathrm{graph}(b_\om)$ have some vertical parts. 
Observe that actually we have the following possibilities:
\begin{enumerate}
\item[(1)] $\pa(- b_\om)(y_0)=\pa a_\om(y_0)=\varnothing$;
\item[(2)] $\pa(-b_\om)(y_0)\not=\varnothing$.
\end{enumerate}
\par
If (1) holds, then at every $x_0\in\cP^{-1}(y_0)\cap\pa\cK$, 
the convex body $\cK$ has only supporting hyperplanes parallel to $e_N$: these are
invariant with respect to the action of $\cA_\om$, so that
their reflections are supporting hyperplanes for $\cK$ at $x_0^\la$ and formula \eqref{necessary} or \eqref{necessary_b} easily follows.
\par
If (2) holds, we have $x_0=x_0^\lambda$ and let us call $\Om=\cK^+_{\la,\omega}\cup\cT_{\la,\omega}(\cK^+_{\la,\omega})$. Then 
\[
N_\cK(x_0)\subset N_\Om(x_0)
\]
and we have $N^+_\Om(x_0)=N^+_\cK(x_0)$, so that
\[
N^-_\Om(x_0)=N_\Om(x_0)\setminus N^+_\Om(x_0) \supset N_\cK(x_0)\setminus N^+_\cK(x_0)=N^-_\cK(x_0).
\]
By observing that $N^-_\Om(x_0)=\cA_\omega(N^+_\Om(x_0))=\cA_\omega(N^+_\cK(x_0))$, \eqref{necessary_b} follows.
\end{proof}
\begin{coro}
Under the same notations of Theorem \ref{th:necessary}, if $\pa\cK$ admits a (unique) unit normal $\nu$ at the point $x_0$, 
then it admits a unit normal at the point $x_0^\la$ too and 
\begin{equation}
\label{necessaryC1}
\cA_{\om}\nu(x_0^\la)=\nu(x_0)\,.
\end{equation}
In particular, if $x_0=x_0^\lambda$ we have $\nu(x_0)\in\om^\bot$.
\end{coro}
\begin{proof}
It is sufficient to observe that in this case 
$$
N_\cK(x_0)=\{\xi\, :\, \xi=t\,\nu(x_0),\, t> 0\}\cup\{0\},
$$
and hence \eqref{necessaryC1} is a consequence of \eqref{necessary} or \eqref{necessary_b}.
\end{proof}
Using Theorem \ref{th:necessary}, we obtain an interesting upper bound on the maximal folding function for a strictly convex domain, in terms of its support function.
\begin{prop}
If $\cK$ is strictly convex, then for every $\omega\in\SSN$
\begin{equation}
\label{paolo}
\cR_\cK(\omega)\leq\max\left\{
\left(\frac{\nabla h_\cK(\xi)+\nabla h_\cK(\cA_\om \xi)}{2}\right)\,\cdot\omega\,:
\,\xi\in\RE^N\setminus\{0\}\right\}.
\end{equation}
\end{prop}
\begin{proof}
First observe that thanks to the $1-$homogeneity of the support function, the maximization problem in \eqref{paolo} can be equivalently settled in $\SSN$.
\par
The strict convexity of $\cK$ implies that $h_\cK\in C^1(\RE^N\setminus\{0\})$ (see \cite{Sc}). Moreover, 
$\na h_\cK(\theta)=x$ for every $\theta\in N_\cK(x)$, with $x\in\pa\cK$.
Thus, with the same notations as in Theorem \ref{th:necessary}, $x_0^\lambda=\nabla h_{\cK}(\theta)$ for every $\theta\in N_\cK(x_0^\lambda)$ and using the condition \eqref{necessary}, we also get $x_0=\na h_\cK(A_\omega(\theta))$; on the other hand,
$a_\om(y_0)=x_0^\la\cdot\omega$ and $b_\om(y_0)=x_0\cdot\omega$, which implies the following
\[
f_\om(y_0)=\frac{\nabla h_\cK(\theta)+\nabla h_\cK(\cA_\om \theta)}{2}\,\cdot\omega\,.
\]
Hence we can conclude by simply applying Theorem \ref{th:brasco}. 
\end{proof}

\begin{oss}
If $\cK$ is symmetric with respect to a hyperplane orthogonal to $\omega$, then equality holds in \eqref{paolo} and both quantities equal $\ovr x_\cK\cdot\omega$. Otherwise,
in general inequality \eqref{paolo} is strict as Figure $1$ informs us.  
\begin{figure}[tbp]
\centerline{
\includegraphics[scale=.4]{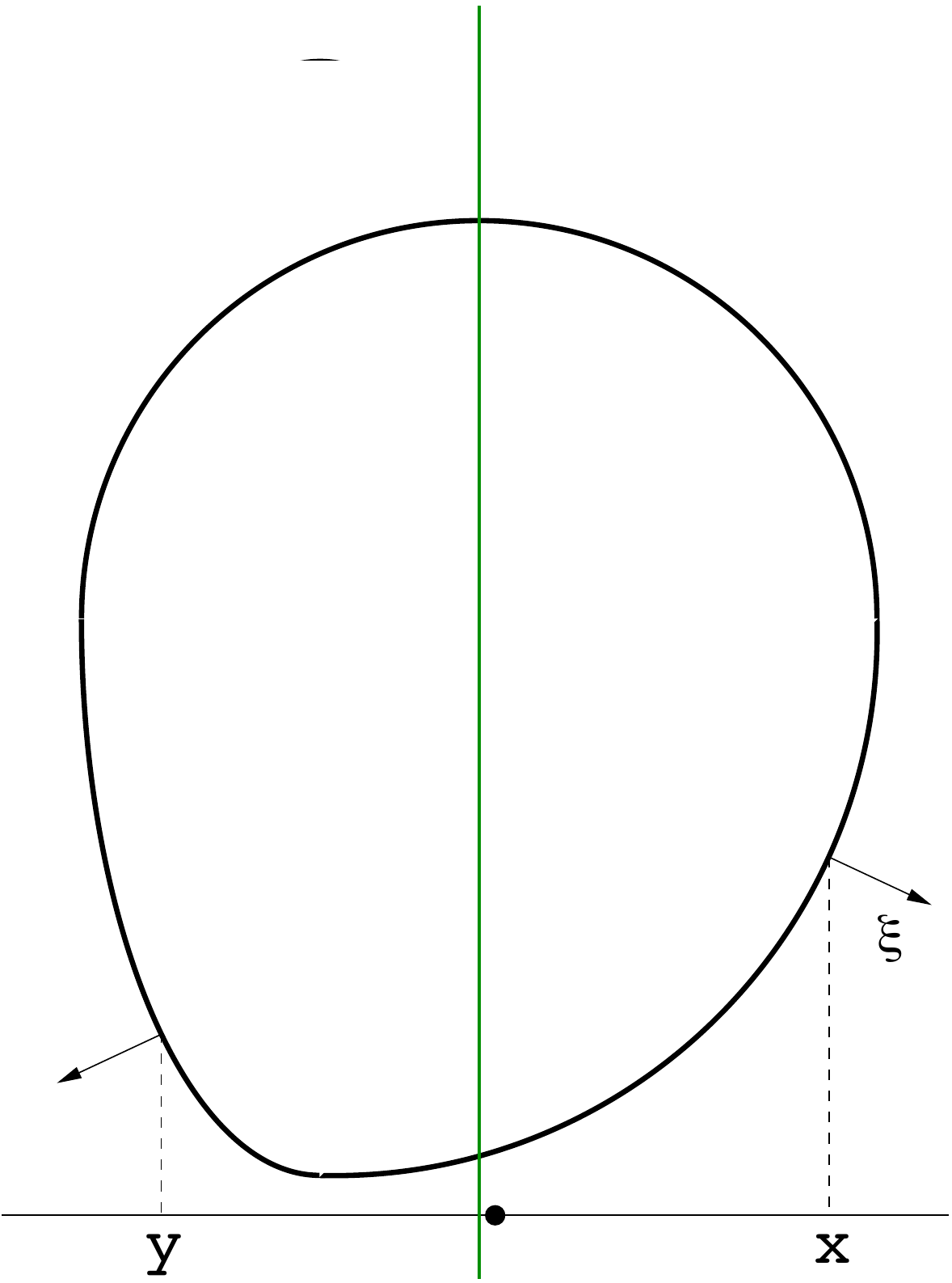}}
\caption{Here $\om=(1,0,\dots,0),$ $x=\na h_\cK(\xi),$ $y=\na h_\cK(\cA_{\om}(\xi));$ 
the intersection of the two straight lines corresponds to $\cR_\cK(\om),$  the dark dot corresponds to 
$\frac12\,(x+y).$}
\end{figure}
\par
Notice that, following an argument similar to that of the proof of Proposition \ref{paolo}, we can in fact give a precise characterization
of the maximal folding function $\cR_\cK$ in terms of the support function $h_\cK$. Precisely the following holds
\begin{equation}
\label{paolo2}
\cR_\cK(\omega)=\max_{\xi\in\Sigma(\omega)}
\frac{\nabla h_\cK(\xi)+\nabla h_\cK(\cA_\om \xi)}{2}\,\,,
\end{equation}
where
$$
\Sigma(\omega)=\left\{\xi\in\RE^N\setminus\{0\}\,:\,\nabla h_\cK(\xi)=\nabla h_\cK(\cA_\om \xi)+\mu\,\omega\,\text{ for some }\mu\in\RE\right\}.
$$
If $\cK$ is not strictly convex (and then $h_\cK$ is not $C^1$) the above formula still remains valid, up to suitably interpreting the gradient of $h_\cK$ as the subdifferential $\partial h_\cK$.
\end{oss}

\section{Numerical examples}
\label{sec:algoritmo}

\subsection{The case of convex polyhedrons}

If $\cK$ is a convex polyhedron, then the conclusions of Theorem \ref{th:brasco} can be improved:
roughly speaking, we can discretize the optimization problem (\ref{lambda}),
by only visiting the projections of the vertices
of $\cK$ on $\om^\perp.$ We begin with the following general result.
\begin{lm}
Let $A$ and $B$ be convex sets such that $A\subseteq B$ and let $x\in\partial A\cap\partial B$. 
\par
If $\pa A$ contains a segment $\ell$ and $x$ belongs to the relative interior of $\ell$, 
then $\ell$ is also contained in $\pa B.$ 
\end{lm}
\begin{proof}
In other words, if $A$ ``touches" $B$ from the interior at $x$ and $x$ is contained in the interior of some segment on the boundary of $A$, then the boundary of $B$ must contain all the segment at the same.
\par
Indeed, let $\pi$ be a support hyperplane to $B$ at $x$ and denote by $\pi^+$ the half-space delimited by $\pi$ and containing $B$; then
$\pi$ is also a support hyperplane to $A$ at $x$ and $A\subseteq\pi^+$. Thus, $\ell\subset\pi^+$ while $x\in\ell\cap\pi\neq\varnothing;$ this implies $\ell\subset\pi,$ since $x$ is not an endpoint of $\ell$, and 
hence $\ell\subset\partial B$.
\end{proof}
\begin{coro}
\label{coro:lm}
Under the same assumptions and notations of Theorem \ref{th:necessary}, if
$x_0$ belongs to the relative interior of a segment $\ell$ contained in $\pa\cK$,
then $\cT_{\la,\om}(\ell)\subset~\pa\cK$ and
\begin{equation}\label{fconstantonsegment}
f_\om(y)=\cR_\cK(\om)\quad\mbox{for every }y\in\cP_\om(\ell)\,.
\end{equation}
\end{coro}
\begin{proof}
The proof follows from the previous lemma by setting $B=\cK$ and $A=\cK\cap\cT_{\la,\om}(\cK)$ and using the definition of $f_\omega$.
\end{proof}
\par

\begin{teo}
\label{th:poligoni}
Let $x_1,\dots,x_s\in\RE^N$ be the vertices of an $N$-dimensional convex polyhedron $\cK\subset\RE^N,$
so that
\[
\cK=\left\{x\in\RE^N\ :\ x=\sum\limits_{i=1}^s\la_i x_i,\ \mbox{with } \sum\limits_{i=1}^s \la_i =1.\ \la_i\in[0,1]\right\},
\]
For a fixed $\omega\in \SSN$, let $f_\om$ be the function
defined by \eqref{f}. 
\par
Then $\cR_\cK$ is the solution of the following discrete optimization problem
\begin{equation}
\label{lambda_pol}
\cR_\cK(\omega)=\max\left\{f_\om(y_j) :\ j=1,\dots,s\right\},
\end{equation}
where $y_j=\cP_\om(x_j)$ is the projection of $x_j$ on $\shk$, for every $j=1,\dots,s$.
\end{teo}
\begin{figure}[tbp]
\centerline{
\includegraphics[scale=.4]{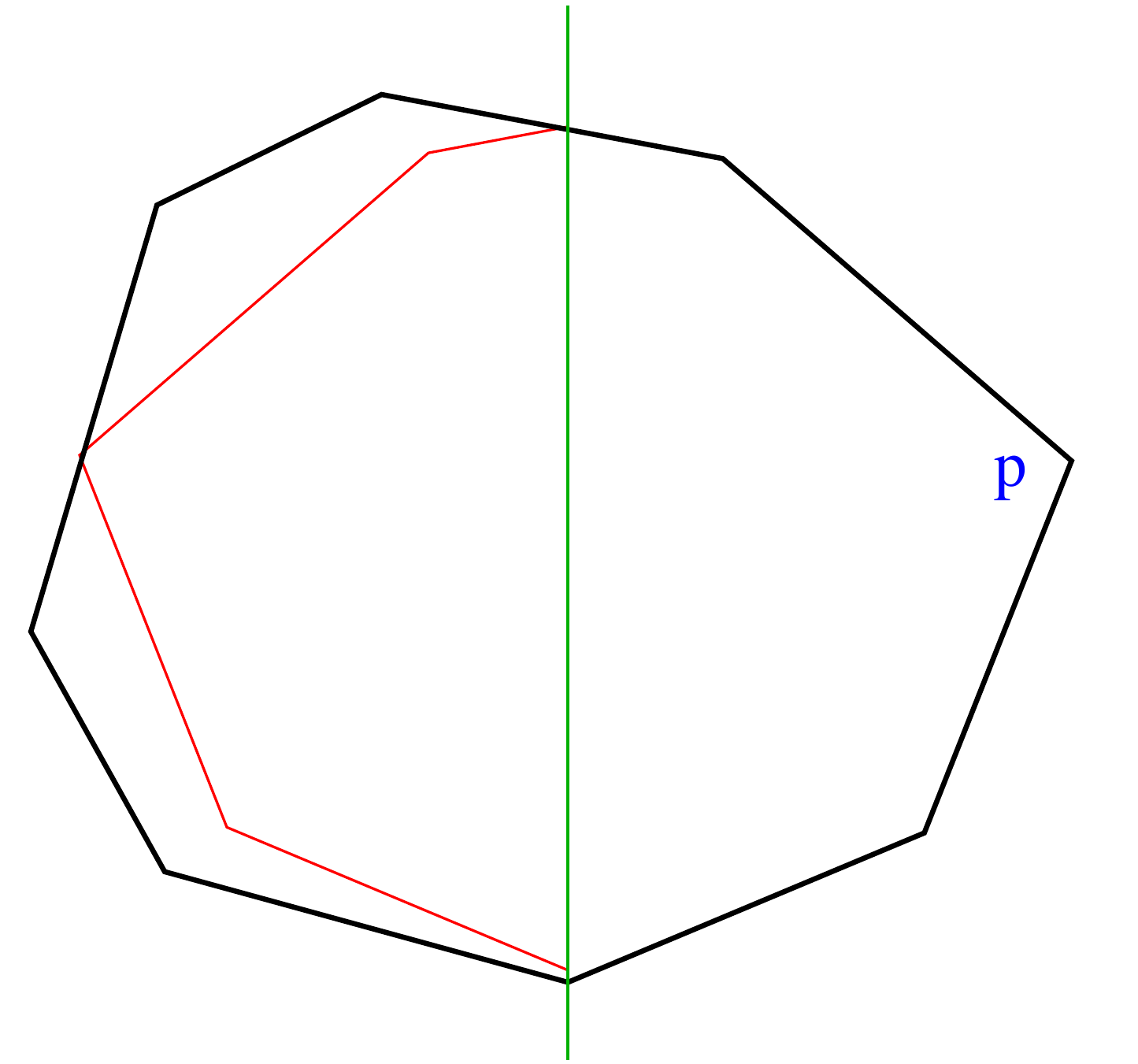}
\includegraphics[scale=.5]{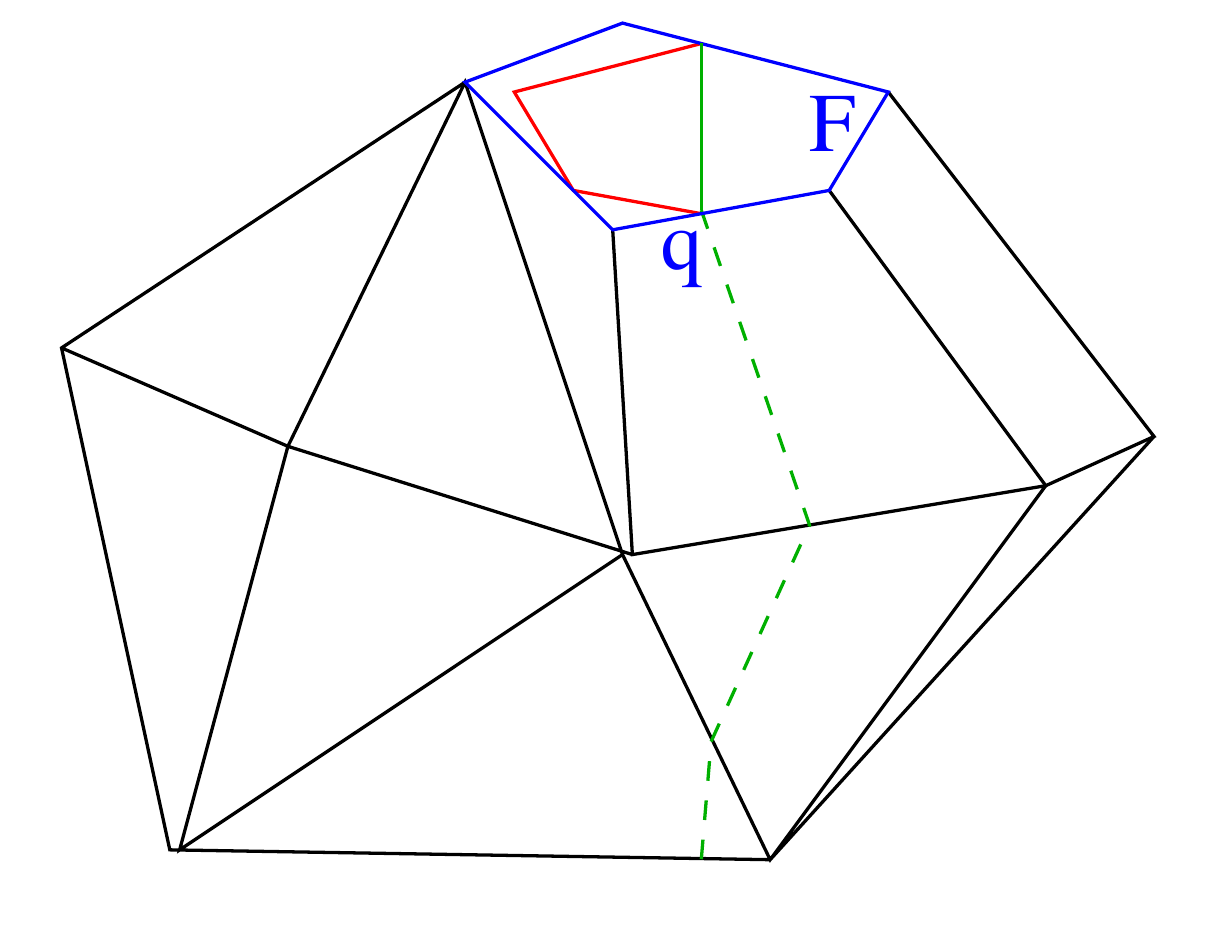}}
\caption{The two cases (i) and (ii).}
\end{figure}
\begin{proof}
By definition \eqref{lambda} and Theorem \ref{th:brasco}, 
we know that the value $\la=\cR_\cK(\om)$ can possibly be achieved when the 
boundary of the reflected cap $\cT_{\la,\om}(\cK^+_{\la,\om})$ is tangent
to that of $\cK$ either
\begin{itemize}
\item[(i)] at a point $p\notin\pi(\la,\om),$ or
\item[(ii)] at a point $q\in\pi(\la,\om)$
\end{itemize}
(see Figure 2).
Thus, the maximum of $f_\om$ is attained at the projection of either $p$ or $q$
on $\shk.$
\par
Now, let $\cK$ be a convex polyhedron. If $p$ is not a vertex of $\cK,$ then $p$
belongs to the relative interior of some $m$-dimensional facet of $\pa\cK,$
with $1\le m\le N-1,$ and hence it belongs to the relative interior of a segment
$\ell$ with (at least) one end at some vertex $v$ of $\pa\cK.$
\begin{figure}[h]
\centerline{
\includegraphics[scale=.5]{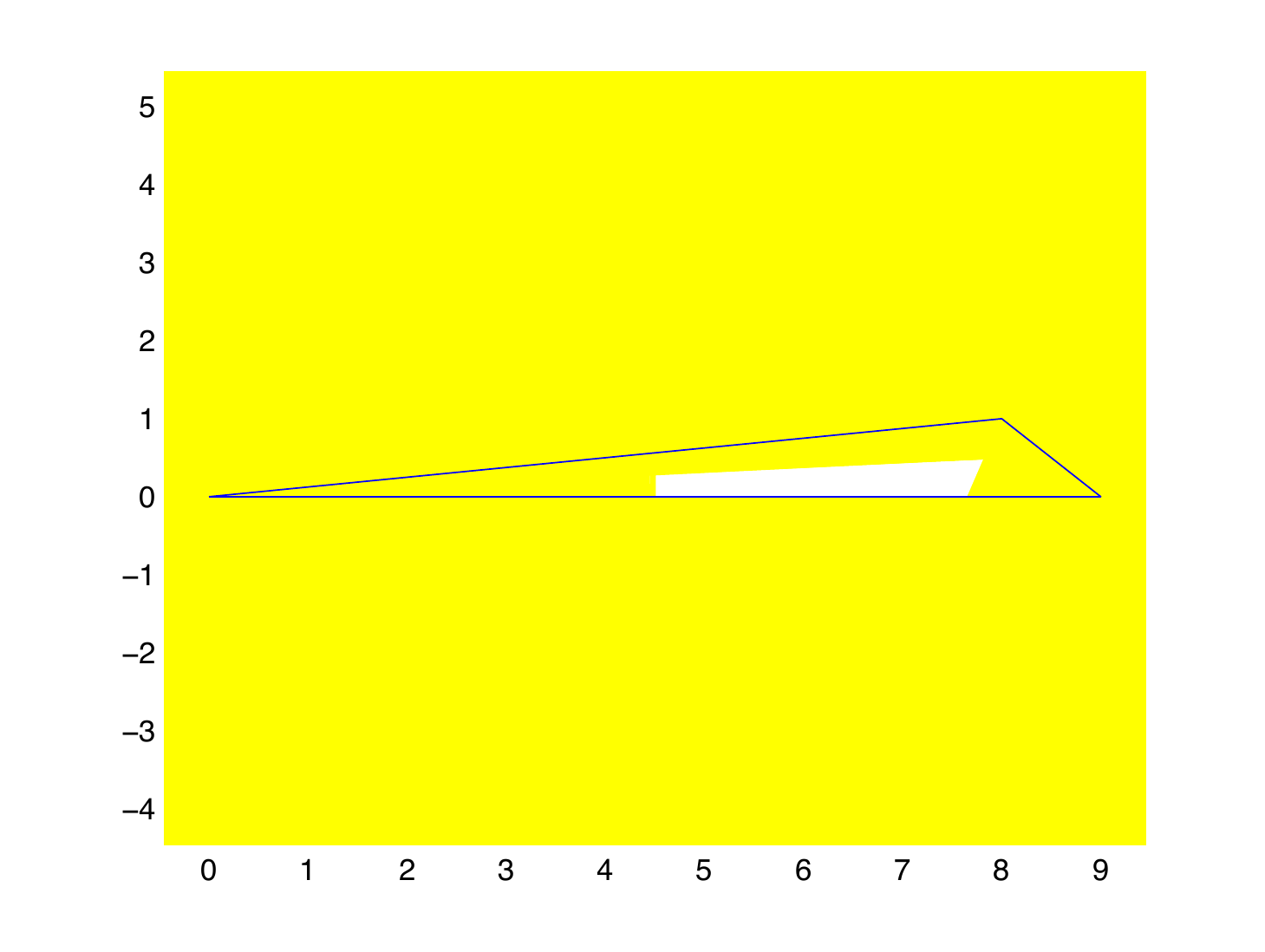}
\includegraphics[scale=.3]{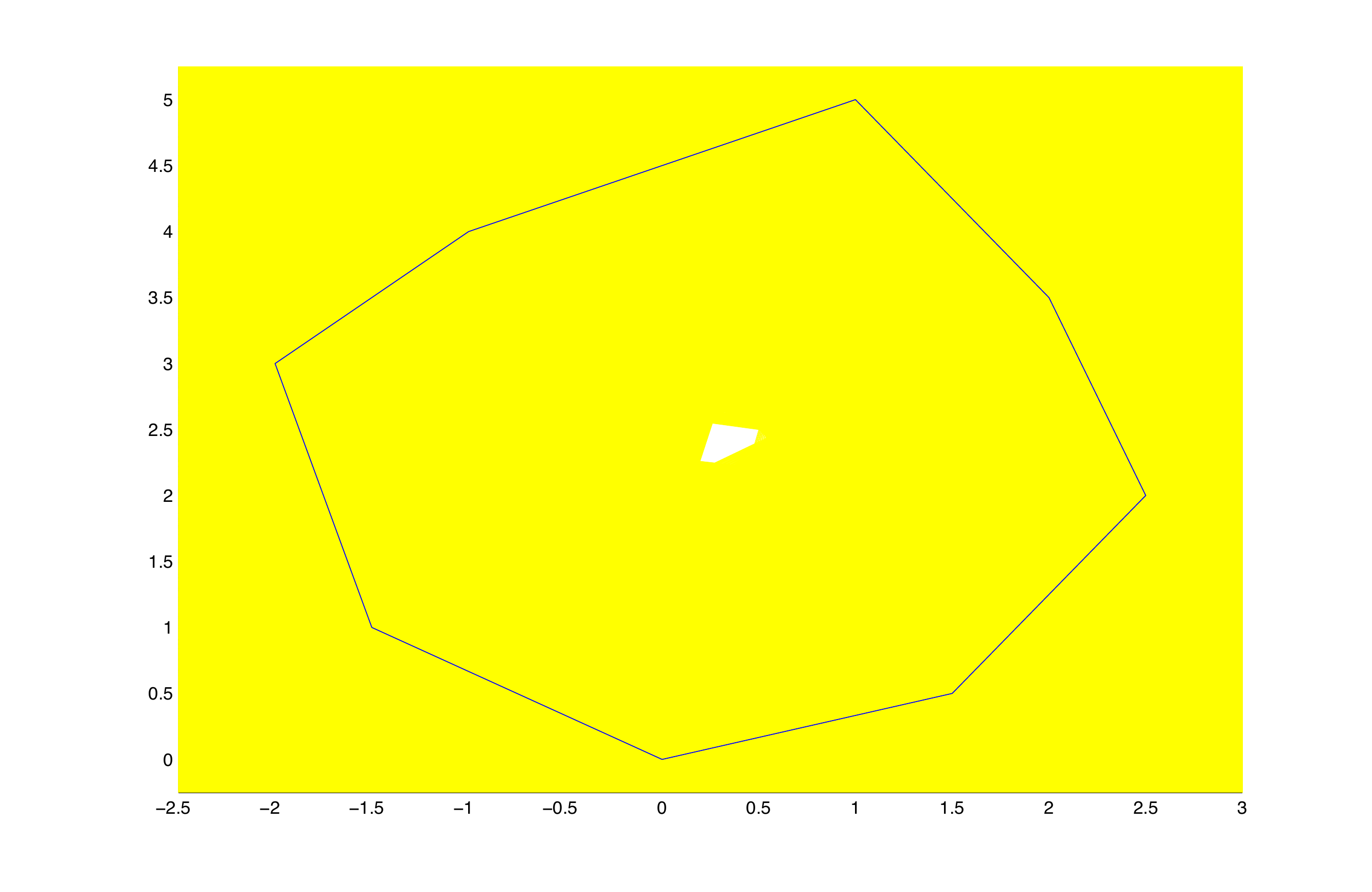}}
\caption{The sets $\cK$ and $\core(\cK)$ in two examples. In the first picture observe that,  by means of \eqref{prova1},
we also know that $x_\infty$ is at a positive (and computable) distance from the boundary of $\cK.$}
\end{figure}
\par
By Corollary \ref{coro:lm}, we then have:
$$
\cR_\cK(\om)=f_\om(\cP_\om(p))=f_\om(\cP_\om(v)).
$$
\par
In case (ii), the $m$-dimensional facet $F$ of $\pa\cK$ containing $q$ must be orthogonal
to the hyperplane $\pi(\la,\om);$ however, the same argument used for case (i) can easily
be worked out in $F.$
\end{proof}
\par
The pictures in Figure 3 show two convex polygons with their relative hearts. The hearts
have been drawn, by using {\sc Matlab}, by an algorithm based on Theorem \ref{th:poligoni}. 

\begin{ack}
The first author was partially supported by the European Research Council under FP7, Advanced Grant n. 226234 ``Analytic Techniques for Geometric
and Functional
Inequalities'', while the second and third authors were partially supported by the PRIN-MIUR grant ``Aspetti geometrici delle equazioni alle derivate parziali e questioni connesse''.
\end{ack}


\begin{thebibliography}{100}

\bibitem{BL} H.~L. Brascamp, E. H. Lieb, {\em On extensions of Brunn-Minkowski and Pr\acc ekopa-Leindler theorems, including inequalities for log-concave functions, and with an application to the diffusion equation}, J. Funct. Anal. 22 (1976), 366--389.
\bibitem{CS} M.~Chamberland, D.~Siegel, {\em Convex domains with stationary hot spots},
Math. Meth. Appl. Sci.  20 (1997) 1163--1169.
\bibitem{Fr} L. E. Fraenkel, Introduction to maximum principles and symmetry in elliptic problems, Cambridge University Press 2000.
\bibitem{FK} P. Freitas, D. Krej\u{c}i\u{r}\'{\i}k, {\em A sharp upper bound for the first Dirichlet eigenvalue and the growth of the Isoperimetric Constant of convex domains}, Proc. American Math. Soc. 136 (2008), 2997--3006.
\bibitem{GNN} B. Gidas, W. M. Ni, L. Nirenberg, {\em Symmetry and related properties via the maximum principle}, Comm. Math. Phys.  68  (1979), no. 3, 209--243.
\bibitem{GJ} D. Grieser, D. Jerison, {\em The size of the first eigenfunction on a convex planar domain}, J. Amer. Math. Soc. 11 (1998), 41--72.
\bibitem{GW} R.~Gulliver, N.B.~Willms, {\em A conjectured
heat flow problem,} In {\rm Solutions,}  SIAM Review 37, (1995)
100--104.
\bibitem{Gu} C. E. Gutierrez, The Monge-Amp\`ere equation, Progress in Nonlinear Differential Equations and their Applications, 44. Birkh\"auser Boston, Inc., Boston, MA, 2001.
\bibitem{Herz} C. S. Herz, {\em Fourier transforms related to convex sets,} Ann. Math. 75 (1962) 81--92.
\bibitem{Ka} B.~Kawohl, {\em A conjectured
heat flow problem,} In {\rm Solutions,}  SIAM Review 37 (1995)
104--105.
\bibitem{Kl} M.~S.~Klamkin, {\em A conjectured heat flow problem,}
In {\rm Problems,} SIAM Review 36 (1994) 107.
\bibitem{Kor} N. Korevaar, {\em Convex solutions to nonlinear elliptic and parabolic boundary value problems,}
Indiana Univ. Math. J. 32 (1983), 603--614.
\bibitem{MS1} R. Magnanini, S. Sakaguchi, {\em On heat conductors with a stationary hot spot},  Ann. Mat. Pura Appl. 183 (2004), 1--23.
\bibitem{MS2} R. Magnanini, S. Sakaguchi, {\em Polygonal Heat Conductors with a Stationary Hot Spot}, J. Anal. Math.
105 (2008), 1--18.
\bibitem{PW} M. H.~Protter, H. F.~Weinberger, Maximum Principles in Differential Equations,  Prentice-Hall, Englewood Cliffs, N. J., 1967.
\bibitem{Sa} S. Sakaguchi, {\it Behavior of spatial critical points and zeros
of solutions of diffusion equations,}  Selected papers on differential equations and analysis,  15--31, Amer. Math. Soc. Transl. Ser. 2, 215, Amer. Math. Soc., Providence, RI, 2005.
\bibitem{Se} J. Serrin, {\em A symmetry problem in potential theory,} Arch. Rational Mech. Anal. 43 (1971), 304--318.
\bibitem{Sc} R. Schneider, Convex bodies: the Brunn-Minkowski theory, Cambridge University Press 1993.
\bibitem{Ta} G. Talenti, {\it Some estimates of solutions to Monge-Amp\`ere type equations in dimension two}, Ann. Scuola Norm. Sup. Pisa Cl. Sci. 8 (1981), 183--230.
\bibitem{Va} S.~R.~S.~Varadhan, {\em On the behaviour
of the fundamental solution of the heat equation with
variable coefficients,} Comm. Pure Appl. Math. 20, (1967), 431--455.

\end{thebibliography}
\end{document}